\documentclass[12pt, reqno]{amsart}

\usepackage{epsfig}

\newcommand{\kahler}{K\"ahler\ }

\newcommand{\R}{{\mathbb R}}

\newcommand{\Z}{{\mathbb Z}}

\renewcommand{\d}{\partial}

\newcommand{\al}{\alpha}

\newcommand{\la}{\lambda}
\newcommand{\ep}{\varepsilon}

\newcommand{\half}{{\frac{1}{2}}}

\newtheorem{theo}{{\sc Theorem}}[section]

\numberwithin{equation}{section}

\theoremstyle{plain}

\newtheorem{lem}[theo]{Lemma}
\newtheorem{prop}[theo]{Proposition}

\newenvironment{rem}{\medskip\noindent{\it Remark:\/} }{\medskip}

\title[The Monge-Amp\`{e}re Equation with Guillemin Boundary Conditions]
{The Monge-Amp\`{e}re Equation with Guillemin Boundary Conditions}

\author{Daniel Rubin}
\address{Department of Mathematics, Columbia University, 2990 Broadway, New York, NY 10027}
\email{rubin@math.columbia.edu}

\begin{document}

\maketitle

\begin{abstract} 
Existence and boundary regularity away from the corners are established for two-dimensional Monge-Amp\`{e}re equations on convex polytopes with Guillemin boundary conditions.
An important step is to derive an expansion in terms of functions $y^n$ and $y^n\log y$ for solutions to equations of the form $\det D^2u(x,y) = y^{-1}$ in a half-ball.   
\end{abstract}

\section{Introduction}

The aim of this paper is to study a seemingly new type of boundary value problem for a real Monge-Amp\`{e}re equation in a convex polytope. 
More precisely,
let $P\subset{\bf R}^n$ be a polytope, and let
\begin{equation}
	P=\cap_{i=1}^N \{l_i(x)>0\},
\end{equation}
be a representation of $P$ as an intersection of half-planes, with $l_i(x)$ an affine function of $x$ for each $i$. We consider the problem of finding
a function $u\in C^0(\overline{P})$ and strictly convex satisfying 
\begin{eqnarray}
\label{cond1}	&&
\det D^2 u(x) ={1\over \varphi(x)}\\	
\label{cond2}	&&
u(x)-\sum_{i=1}^N l_i(x)\log l_i(x) \in C^\infty(\overline{P})
\end{eqnarray}
where the given function $\varphi $ on $P$ is of the form
\begin{equation}\label{cond5}
\varphi(x)=  h(x)\prod_{i=1}^N l_i(x)
\end{equation}
with $h(x)\in C^\infty(\overline{P})$, $0<h(x)$. Boundary conditions of the form (\ref{cond2}) are called Guillemin boundary conditions.

\medskip
The motivation for this problem comes from toric geometry, and particularly Abreu's equation \cite{A}, which is the equation for a K\"ahler metric of constant scalar curvature on a toric variety. We shall say more about this later, but for the moment, we note that an essential part of the problem is the particular form of the boundary condition, and the fact that the equation takes place on a polytope. For example, a naive version of the problem on a strictly convex domain $D\subset{\bf R}^2$ with boundary function $d(x)$ of the form $u-d\,\log\,d\in C^3(\overline{D})$, $\det D^2 u(x)=O(d^{-1})$, would have no solution, since the boundary asymptotics for $u$ would imply that $\det D^2 u(x)\sim d^{-1}\log\,d^{-1}$ near $\partial D$. Thus the polytope features of the problem have to be fully taken into account, and they play indeed a major role in our results which we describe next.

\medskip
Let $n_i$ be the unit, inward-pointing normal to the face $l_i(x)=0$ of the polytope $P$, and set $l_i(x) = n_i \cdot x -\lambda$. 
Any two vectors $n_i$ and $n_k$ define a matrix $n_i^\alpha n_k^\beta$, the determinant of which is the area of the parallelogram spanned by $n_i$ and $n_k$. 
We also denote the $N$ vertices of the polytope $P$ by $v_i$,
$1\leq i\leq N$, with $v_i$ the intersection of the faces $l_{i-1}=0$ and $l_i=0$.Then

\begin{theo}
\label{maintheo}
Let $P $ be a convex polytope in  $\R^2$, and consider the problem
(\ref{cond1}, \ref{cond2}) where $h\in C^\infty(\overline{P})$ and $h(x)>0$. 

{\rm (a)} If the equation admits a solution $u$ which is convex in $P$ and satisfies the boundary condition in (\ref{cond1}, \ref{cond2}), then the given function $h(x)$ must satisfy
\begin{eqnarray}
	\label{cond3}h(v_i) &=& \left(\det(n_{i-1}n_i)^2\prod_{j \neq i-1,i} l_j(v_i)\right)^{-1}
\end{eqnarray}

{\rm (b)} Conversely, assume that the given function $h(x)$ satisfies (\ref{cond3}).
Then there exists $\al>0$ such that for each choice of values $\{\al_i\}_{i=1}^N$, $a_i \in \R$, there is a unique solution $u \in C^\al(\overline{P})$ to the equation
(\ref{cond1}), satisfying the following boundary condition
\begin{equation}
u - \sum_{i=1}^N l_i(x)\log l_i(x) \in C^\infty(\overline{P}\setminus
\{v_1,\cdots,v_N\}),
\quad
{\rm and}
\ \ u(v_i)=\alpha_i,
\quad
1\leq\alpha\leq N.
\end{equation}
\end{theo} 

At this moment, the regularity of the solution at the corners is still open. 

\medskip
We discuss briefly some of the main steps in the proof of Theorem \ref{maintheo}. A key observation is that, if a solution $u$ exists, then 
its restriction to the edge $e_i$ is the solution of the following second-order ODE along the edge $e_i$,
\begin{equation}
\partial_{T_i}^2u = |n_i|^2/\varphi_{n_i}
\end{equation}
Combined with the assigned values of $u$ at the vertices $v_i$, this equation determines completely the restriction of $u$ to the boundary $\partial P$ of the polytope. Thus, we can obtain $u$ by solving the Monge-Amp\`ere equation (\ref{cond1}) with this given Dirichlet condition. Because the right hand side of the equation (\ref{cond1}) blows up near the boundary, and because the domain is a polytope, the solution does not appear to have been written down previously in the literature. However, we show in section \S 3 that the methods of Cheng-Yau \cite{CY} can be suitably extended to produce a generalized solution.

\smallskip
The remaining issue is the regularity. The $C^\alpha$ regularity on $\overline{P}$ is established by constructing suitable barrier functions. The regularity and asymptotic expansion at the edges are modeled on the following problem
\begin{equation}\label{modelFace}
\det D^2u(x',x_n) = x_n^{-1}
\end{equation}
near the interior of a face $\{x_n=0\}$, and
\begin{equation}\label{modelCorner}
\det D^2u = (x_1...x_n)^{-1}
\end{equation}
near a corner. The equation (\ref{modelFace}) is a limit case of the equations studied by Daskolopoulos and Savin in \cite{DS} (and in \cite{S} in higher dimensions) of the form
\begin{equation}
	\det D^2v(x,y) = y^\al \text{ in $B_1$, $\al>-1$}
\end{equation}
for which they obtained the behavior of the solution
\begin{equation}
	v(x,y) = \frac{1}{2a}x^2 + \frac{a}{(\al+2)(\al+1)}|y|^{2+\al} + O\left((x^2+|y|^{2+\al})^{1+\delta}\right)
\end{equation}
for some $a>0$, in a neighborhood of the origin. The case of exponent -1 presents a new difficulty from the fact that solutions with quadratic growth on the flat boundary have infinite normal derivative. In our case, we need to combine the techniques of \cite{DS} with a careful analysis of the partial Legendre transform of $u$ and of the Monge-Amp\`ere equation.

\medskip

We say now a few words about the motivation from toric geometry. Let $X$ be a toric variety of dimension $n$. Then its image under the moment map is a polytope $P$ in ${\bf R}^n$, and a toric K\"ahler metric on $X$ can be determined by 
a function $u:\bar{P}\rightarrow \R^n$ called the \textit{symplectic potential}, which is the  Legendre transform of the \kahler potential in the open torus. 
Guillemin \cite{G} showed that the symplectic potential of a smooth toric variety satisfies
\begin{equation}
	u(x)=\sum_{i=1}^N l_i(x)\log l_i(x) + f(x), \text{ $f \in C^\infty(\overline{P})$, $u$ convex in $\overline{P}$}
\end{equation}
where the affine functions $l_i(x)$ defining the faces of $P$ have been appropriately normalized. As shown by Abreu \cite{A}, the K\"ahler metric is an extremal metric if and only if its symplectic potential $u$ satisfies the so-called Abreu equation
\begin{equation}
\label{Abreu}
	\sum_{i,j=1}^n \frac{\d^2 u^{ij}}{\d x_i \d x_j} = -A,
\end{equation}
where $(u^{ij})$ is the inverse of the Hessian $(u_{ij})$
where $A$ is an affine function. The metric is of constant scalar curvature when $A$ is constant. The Abreu's equation is clearly equivalent to the following system of two second-order elliptic equations for the two unknowns $(u,\varphi)$,
\begin{eqnarray}
\label{Abreu1}
	\det D^2u = \varphi^{-1} \\
	U^{ij}\varphi_{ij} = -A
\end{eqnarray}
where $(U^{ij})$ is the cofactor matrix of the Hessian of $u$. From the boundary condition, it follows that the function $\varphi(x)$ must vanish to first order along each face. 

\smallskip
The existence of a metric of constant scalar curvature, and hence the solvability of Abreu's equation, has been shown by Donaldson in dimension $n=2$ to be equivalent to the K-stability of the toric variety $X$
\cite{D1}. The same statement is expected to hold in all dimensions, and is known as the Tian-Yau-Donaldson conjecture \cite{D2} (see also \cite{PS} for a survey). Donaldson also gave interior estimates for Abreu's equation, using in part works of Trudinger-Wang \cite{TW} on similar equations arising from the affine Plateau problem. Donaldson's results were subsequently extended by Chen, Li, and Sheng \cite{CLS}, who solved the problem of general prescribed curvatures in dimension two, and also by Chen, Han, Li, and Sheng \cite{CHLS} giving interior estimates for all dimensions.

\smallskip
But even in dimension $n=2$, a major question is to understand the singularities of the solutions of Abreu's equation in general. This is a difficult problem since Abreu's equation is of fourth-order, and it is natural as a first step to explore separately the two second-order equations appearing in (\ref{Abreu1}). The second equation is a linearized Monge-Amp\`ere equation of the type studied by Caffarelli-Gutierrez \cite{CG}. The first equation, together with the Guillemin boundary conditions, is a new type of boundary value problem for the Monge-Amp\`ere equation, which may be of independent interest and is the equation studied in the present paper.

\medskip

The paper is organized as follows: In section 2, we explain the setup and derive the necessary conditions on the right-hand side, as well as the boundary equation, which we solve to give Dirichlet data compatible with the Guillemin boundary conditions. In section 3, we give a Perr\'{o}n's method argument to solve the Dirichlet problem, ensuring that there exists a solution in the polytope which is H\"{o}lder continuous up to the boundary. In section 4, the main part of the paper, we deal with the behavior of the solution near an edge, establishing that under the precise boundary relation, the solution goes like $l_i\log l_i + f$, with $f$ smooth. This completes the proof of Theorem \ref{maintheo}. For the most part we work exclusively in dimension two. This restriction is mainly for simplicity of computation in sections two and three, where the results have clear extensions to higher dimensions, but is essential in section four where we take the partial Legendre transform of the Monge-Amp\`{e}re equation to yield a quasilinear equation.

\medskip
\textbf{Acknowledgements:} I would like to thank my advisor D.H. Phong for his guidance and encouragement, and I am also grateful to Ovidiu Savin and Connor Mooney for many helpful conversations.

\section{Consequences of the Guillemin boundary conditions}

In general, an asymptotic expansion for the solution $u$ near the boundary of a domain will put some constraints on the boundary behavior of ${\rm det}\,D^2u$. In the case of Guillemin boundary conditions on a polytope, these constraints turn out to be quite powerful. This is the contents of Theorem \ref{maintheo}, part (a), which we reformulate as the following separate proposition for convenience:

\begin{lem}
Let $u\in C^0(\overline{P})\cap C^\infty(\overline{P}\setminus
\{v_1,\cdots,v_N\})$ be a function which satisfies the Guillemin boundary condition
(\ref{cond2}) on $\overline{P}\setminus \{v_1,\cdots,v_N\}$ in the sense that
\begin{equation}
u(x)-\sum_{i=1}^N l_i(x)\log l_i(x)
\in C^0(\overline{P})\cap C^\infty(\overline{P}\setminus \{v_1,\cdots,v_N\}).
\end{equation} 
Then
\begin{equation}
	\det D^2u ={1\over h(x)\prod_{i=1}^N l_i(x)},
\end{equation}
where $h(x)$ is a function which is in $C^0(\overline{P})\cap C^\infty(\overline{P}\setminus \{v_1,\cdots,v_N\})$ 
and satisfies $0<h(x)$. When the full Guillemin boundary condition (\ref{cond2}) holds, then $h\in C^\infty(\overline{P})$. Furthermore,
\begin{equation}
	h(v_k) = \frac{1}{\det(n_{k-1}n_k)^2\prod_{j\neq k-1,k}l_j(v_k)}.
\end{equation}

\end{lem}

\begin{rem}
In the case when the polygon is Delzant, the integral inner normal vectors of two adjacent edges form a basis of $\Z^2$, so $\det(n_{k-1}n_k)^2=1$.
\end{rem}

\begin{proof}
This result and its extension to higher dimension is due to Donaldson in \cite{D2}. We perform the calculation globally in two dimensions to obtain the right constant; however, the main point is that the values at the vertices do not depend on the potential $u$.
\smallskip

We have
\begin{equation}
	D^2u = \begin{pmatrix} f_{xx}+\sum\frac{(n_i^x)^2}{l_i} & f_{xy}+\sum\frac{n_i^x n_i^y}{l_i}\\ f_{xy}+\sum\frac{n_i^x n_i^y}{l_i} & f_{yy} + \sum\frac{(n_i^y)^2}{l_i}\end{pmatrix},
\end{equation}
so
\begin{align*}
\det D^2 u =& \sum_{i,j} \frac{(n_i^x)^2(n_j^y)^2}{l_i l_j}
             -\sum_{i,j}\frac{n_i^x n_i^y n_j^x n_j^y}{l_i l_j}\\
          & + \sum_i \frac{f_{xx}(n_i^y)^2 + f_{yy}(n_i^x)^2
             		 -2f_{xy}n_i^x n_i^y}{l_i}
             + \det D^2f\\
           =& \frac{1}{\prod_k l_k}\left[\sum_{i\neq j}\left((n_i^x)^2(n_j^y)^2- 
                  n_i^x n_i^y n_j^x n_j^y\right)\prod_{q\neq i,j}l_q \right. 
                  \\
                & \left. + \sum_i (f_{xx}(n_i^y)^2 + f_{yy}(n_i^x)^2
             		 -2f_{xy}n_i^x n_i^y)\prod_{j\neq i}l_j
                  + \det D^2f \prod_k l_k \right] .\\
\end{align*}
The term in the brackets is the function $1/h$. When evaluating $h$ at the vertex $v_k$, both $l_{k-1}$ and $l_{k}$ are zero, so only the terms from the first sum with $i=k-1$, $j=k$, and $i=k$, $j=k-1$ are nonzero, and therefore
\begin{align*}
\frac{1}{h(v_k)} =& \left(\left((n_{k-1}^x)^2(n_k^y)^2-n_{k-1}^x n_{k-1}^y n_k^x n_k^y\right)+\left((n_k^x)^2(n_{k-1}^y)^2-
                  n_k^x n_k^y n_{k-1}^x n_{k-1}^y\right)\right)\prod_{q\neq k-1,k}l_q(v_k)\\
                  =& \left(n_{k-1}^x n_k^y - n_k^x n_{k-1}^y\right)^2\prod_{q\neq k-1,k}l_q(v_k)\\
                  =& \det(n_{k-1}n_k)^2\prod_{q\neq k-1,k} l_q(v_k).
                  \end{align*}

\end{proof}

Now we determine the restrictions on the Dirichlet boundary data.

\begin{lem}
Let $u$ be a function which satisfies the Guillemin boundary condition
(\ref{cond2}) near the boundary of the polytope $P$.
Set $\det D^2u = 1/\varphi$. Then 
\begin{equation}\label{bdrycond}
U^{n_i n_i}\varphi_{n_i}|_{l_i=0} = |n_i|^2 
\end{equation}
where the limit is taken as $x$ approaches any point on the edge away from the vertices.
\end{lem}
\begin{proof}
\begin{eqnarray*}
	U^{n_k n_k} &=&\begin{pmatrix} n_k^x & n_k^y \end{pmatrix} \begin{pmatrix} f_{yy} + \sum\frac{(n_i^y)^2}{l_i} & -f_{xy}-\sum\frac{n_i^x n_i^y}{l_i} \\ -f_{xy}-\sum\frac{n_i^x n_i^y}{l_i} & f_{xx}+\sum\frac{(n_i^x)^2}{l_i} \end{pmatrix} \begin{pmatrix} n_k^x \\ n_k^y \end{pmatrix} \\
	 &=& (f_{xx}(n_k^y)^2+f_{yy}(n_k^x)^2-2f_{xy}n_k^x n_k^y)\\ &&+\left((n_k^x)^2\sum_i\frac{(n_i^y)^2}{l_i} + (n_k^y)^2\sum_i\frac{(n_i^x)^2}{l_i}-2n_k^x n_k^y \sum_i\frac{n_i^x n_i^y}{l_i}\right)\\
	 &=& (f_{xx}(n_k^y)^2+f_{yy}(n_k^x)^2-2f_{xy}n_k^x n_k^y)\\ &&+\left((n_k^x)^2\sum_{i\neq k}\frac{(n_i^y)^2}{l_i} + (n_k^y)^2\sum_{i\neq k}\frac{(n_i^x)^2}{l_i}-2n_k^x n_k^y \sum_{i\neq k}\frac{n_i^x n_i^y}{l_i}\right)\\
\end{eqnarray*}
since the terms with $i=k$ cancel in the sum in parentheses. Also we have
\begin{equation}
	\varphi_{n_k} = h \sum_j (n_k \cdot n_j)\prod_{i\neq j}l_i + D_{n_k}h\prod_j l_j,
\end{equation}
so
\begin{eqnarray*}
	U^{n_k n_k}\varphi_{n_k}|_{l_i=0} &=& h |n_k|^2\prod_{i\neq k}l_i\Bigg[(f_{xx}(n_k^y)^2+f_{yy}(n_k^x)^2-2f_{xy}n_k^x n_k^y)\\ &&+\left((n_k^x)^2\sum_{i\neq k}\frac{(n_i^y)^2}{l_i} + (n_k^y)^2\sum_{i\neq k}\frac{(n_i^x)^2}{l_i}-2n_k^x n_k^y \sum_{i\neq k}\frac{n_i^x n_i^y}{l_i}\right)\Bigg]\\ 
	&=& h|n_k|^2h^{-1}|_{l_k=0}\\
	&=& |n_k|^2.
\end{eqnarray*}
\end{proof}
The boundary equation (\ref{bdrycond}) was exploited in \cite{LS} in connection with the variational approach to Abreu's equation. In that context, this relation followed from the Euler-Lagrange equation satisfied by a minimizer, but in our context it follows directly from the Guillemin boundary conditions by computation as above.

\medskip
A very important consequence of the previous lemma is that, up to the values of the solution $u(x)$ at the vertices $v_1,\cdots,v_N$, the Guillemin boundary conditions determine the boundary values of $u$. Indeed, 
in dimension two, the cofactor matrix entry $U^{nn}$ is equal to a constant multiple of the second tangential derivative along the edge. We may then interpret this lemma as giving a second-order ODE on each edge for $u$. We parametrize the $i$-th edge, where $\{l_i(x) = 0\}$, by 
$$x = v_i + tT_i.$$

\begin{lem}
Let $u\in C^2([0,L])$ solve $u_{tt} = \frac{h(t)}{t(L-t)}$ where $h(t)$ is smooth and positive on $[0,L]$. Then 
\begin{equation}
u(t) = h(0)t\log t +h(L)(L-t)\log(L-t) + v(t),
\end{equation} 
where $v$ is smooth on $[0,L]$. The function $u(t)$ is determined uniquely by its boundary values $u(0)$ and $u(L)$.
\end{lem}

\begin{proof}
By Taylor expanding $h$ at 0, we see that $h(0)t\log t$ accounts for the singularity there, and similarly at the other endpoint. What remains on the right-hand side is smooth, and can be integrated twice to obtain $v$. This proves the desired identity. The second statement is easy, since two solutions of this second order ODE must differ by an affine function of $t$.
\end{proof}

The following statement now follows readily from the previous two lemmas:

\begin{lem}
\label{Dirichlet}
Let $u\in C^0(\overline{P})$ be a strictly convex function on $P$ satisfying the equation (\ref{cond1}) and the Guillemin boundary condition (\ref{cond2}). Let $\alpha_i=u(v_i)$ be the values of $u$ at the vertices $v_i$, and define the function $\hat u\in C^0(\partial P)$ as the unique solution on each edge $e_i$ of the equation
\begin{equation}\label{bdrycond}
\partial_{T_i}^2 \hat u=|n_i|^2\varphi_{n_i}^{-1},
\quad
\hat u(v_i)=\alpha_i,
\ \ 
\hat u(v_{i+1})=\alpha_{i+1}.
\end{equation}
Then the function $u$ is a solution of the Dirichlet problem,
\begin{equation}
\label{Dir}
{\rm det}\,D^2u={1\over\varphi(x)}
\ \ {\rm on}\ P,
\qquad
u_{\vert_{\partial P}}=\hat u.
\end{equation}
\end{lem}
\medskip

\section{Solution of Dirichlet Problem}

We turn now to the proof of Theorem \ref{maintheo}, part (b). In view of Lemma \ref{Dirichlet}, we shall define the desired solution $u$ as the solution of the Dirichlet problem (\ref{Dir}), where the Dirichlet data $\hat u$ is specified by the values $\alpha_i$ and the function $\varphi$.
\smallskip

As a first step, we will first show the existence and uniqueness of generalized solutions to equations of this type, following closely the Perron's method approach of Cheng-Yau. The only new difficulty is that our domain is a polygon, hence not strictly convex. This has consequences for the allowable boundary data and the regularity at the boundary.
\medskip

Recall the definition of an Alexandroff solution: Let $u$ be a convex function on a domain $\Omega \in \R^n$. For each point $x \in \Omega$, let $B(x) = \{p_1,...,p_n\}$ be the set of hyperplanes $x_{n+1} = \sum p_i x_i +b$ passing through $(x,u(x))$ and lying below the graph of $u$. To the function $u$ we associate the measure $\mu(u)$, where $\mu(u)(E) = |B(E)|$. Additionally we define for $\varphi \in C(\Omega)$ the measure of $u$ with weight $\varphi$ to be $$\mu_\varphi(u,E) = \int_E \varphi(x)d\mu(u,x)$$ for any Borel subset $E$ of $\Omega$. If $\mu_\varphi(u)=\mu$ where $u$ is a convex function on $\Omega$ and $\mu$ is a Borel measure, then $u$ is a generalized solution of $\det D^2u = (1/\varphi)\mu$. In our equation, we take the Borel measure $\mu$ to be the ordinary Lebesgue measure. We make repeated use of the following two lemmas, which are now standard:

\begin{lem}
Let $u_i$ be a sequence of convex functions defined on $\Omega$ which converges uniformly on compact sets to a convex function $u$. Then $\mu(u_i)$ converges to $\mu(u)$ weakly.
\end{lem}

\begin{lem}\label{comparison}
Let $u_1$ and $u_2$ be two convex functions defined on a domain $\Omega$ with $u_1=u_2$ on $\d \Omega$ and $u_1\geq u_2$ on $\Omega$. Then $\mu(u_2)\geq \mu(u_1)$
\end{lem}

First we use a basic proposition taken directly from \cite{CY}, whose proof we include for convenience.

\begin{prop}\label{basicprop}
Let $\Omega$ be a polytope in $\R^n$ with vertices $\{v_1,...,v_n\}$. Suppose $\varphi \in C(\Omega)$, $\varphi\geq 0$, and for all compact sets $K$ in $\Omega$ there is a constant $c>0$ such that $\inf_{x\in K} \varphi(x) \geq c$. Let $f$ be a function which is affine linear on $\d \Omega$, that is, $f:\d \Omega \rightarrow \R$ such that $$f(\sum \lambda_i v_i) = \sum \lambda_i a_i$$ for any $\lambda_i \geq 0$ and $\sum \lambda_i = 1$, $a_1,...,a_N \in \R$. Then for any Borel measure $\mu$ with compact support $K$ contained in $\Omega$ and $\mu(\Omega)<\infty$, there exists a unique continuous convex function $u$ on $\bar{\Omega}$ such that $\mu_\varphi(u)$ realizes $\mu$ and $u=f$ on $\d \Omega$.
\end{prop}

\begin{proof}
First take $\mu$ to be a sum of point masses $\mu = \sum_{i=1}^m c_i \delta_{x_i}(x)$, with $c_i>0$. Let $\mathcal{F}$ denote the family of piecewise linear convex functions $w$ with $w=f$ on $\d \Omega$ with $\mu_G(w) \leq \mu$ (so the vertices of the polyhedron defined by the graph of $w$ are a subset of the $\{x_i\}$). $\mathcal{F}$ is non-empty since the convex hull of the data at the vertices is the graph of a piecewise linear function equal to $f$ on the boundary and with mass equal to 0.
\smallskip

Set $\phi(w) = \sum_{i=1}^m w(x_i)$. Then $\phi$ is bounded below in terms of $\inf d(x_i,\d\Omega)$, $\inf \varphi(x_i)$, and $\mu(\Omega)$ by the Alexandroff maximum principle. In the topology of uniform convergence, $\mathcal{F}$ is compact, and $\phi$ is continuous, so $\phi$ achieves its minimum at some $\bar{w}\in \mathcal{F}$.
\smallskip

Then $\mu_G(\bar{w}) = \mu$: If not, suppose the mass of $\mu_\varphi(\bar{w})$ is strictly less than $c_1$ at $x_1$. Then there exists $\ep>0$ such that the piecewise linear function $\hat{w}$ obtained from $\bar{w}$ by lowering its value at $x_1$ by $\ep$, that is, the function whose graph is the convex hull of $(x_1,\bar{w}(x_1)-\ep), \{(x_i,\bar{w}(x_i))\},\{(v_j,f(v_j))\}$, also has mass less than $\mu$. But $\phi(\hat{w})<\phi(\bar{w})$, so we get a contradiction.
\smallskip

For a general Borel measure $\mu$ with compact support K, we let $\mu_i$ be a sequence of sums of point masses converging weakly to $\mu$, and $u_i$ the sequence of piecewise linear functions constructed as above with $\mu_\varphi(u_i)=\mu_i$. The functions $u_i$ are uniformly bounded below as before in terms of $d(K,\d\Omega)$, $\inf_K \varphi$, $\mu(\Omega)$, so they converge uniformly on compact subsets to $u$ with $\mu_\varphi(u)=\mu$. Since also the $u_i$ have bounded Lipschitz norm in terms of the boundary data and $d(K,\d\Omega)$, $\inf_K \varphi$, $\mu(\Omega)$, $u$ has bounded Lipschitz norm and $u=f$ on $\d \Omega$.  
\end{proof}
\medskip

Now we want to solve with more general boundary data. Since we remain in the setting of polygons, which are not strictly convex, we must insist that the boundary data is convex on each face.

\begin{prop}
Let $\Omega$ be a polygon in $\R^2$. Let $f:\d \Omega \rightarrow \R$ be convex on each edge and continuous, and $\varphi$, $\mu$ as in Proposition \ref{basicprop}. Then there is a unique continuous convex function $u$ on $\bar{\Omega}$ such that $\mu_\varphi(u)$ realizes $\mu$ and $u=f$ on $\d \Omega$. 
\end{prop}

\begin{proof}
We approximate the solution of this problem with the solutions of Proposition \ref{basicprop} by taking a sequence of sets of vertices where $A_1 = \{v_1,...,v_N\}$, the set of vertices of $\Omega$, and each set $A_n$ contains the midpoint of any vertices from the previous set $A_{n-1}$ lying on the same edge of $\Omega$. The same proof shows that for each set of vertices $A_n$, the Dirichlet problem can be solved in $\Omega$ for a continuous convex function $u_n$ with boundary data equal to $f(x_i)$ at each point $x_i$ of $A_n$ and linear on the edges in between, since we can still form the non-empty family of piecewise linear convex functions in $\overline{\Omega}$ matching the boundary data with mass less than a sum of point masses. 

\smallskip
The $u_n$ are uniformly bounded below and decreasing, and thus converge to a continuous solution $u$ with $u=f$ on $\d \Omega$. Note that the convergence is not necessarily Lipschitz in the corners since the boundary data need not be Lipschitz there. 
\end{proof}

Now we must allow the measure $\mu$ to have support up to the boundary.

\begin{theo}\label{dirichletprob}
Let $P=\cap\{l_i> 0\}$ be a polygon in $\R^2$. Let $f:\d P \rightarrow \R$ be continuous and convex, with second tangential derivatives $f_{tt}<C/d$, where $d$ is the distance to the nearest vertex. Let $\varphi$ be a smooth function such that there exist positive constants $a,A$ where $a \prod l_i \leq \varphi \leq A \prod l_i$, and $\mu$ a finite Borel measure. Then there exists a unique continuous convex function $u$ on $\overline{P}$ such that $\mu_\varphi(u)$ realizes $\mu$. Moreover, for all $0<\alpha<2/(N+1)$, $u \in C^\alpha(\overline{P})$, with norm bounded in terms of $C$, $\mu(P)$, $a$, and $A$. 
\end{theo}

\begin{proof}
Let $\{h_i\}$ be an increasing sequence of cutoff functions with compact support in $P$ such that $0\leq h_i \leq 1$ and $h_i$ is identically equal to 1 on any compact subset of $P$ for $i$ sufficiently large. From the previous proposition, we have a sequence of functions $u_i$ such that $u_i=f$ on $\d P$ and $\mu_\varphi(u_i)$ realizes $h_i \mu$. By the lemma, the $u_i$ are a decreasing sequence of functions. 
	To establish the existence of a limit with the stated boundary regularity, we must find a lower barrier. This step is more difficult for a polygon than in the case of a uniformly convex domain because of the lack of a $C^2$ convex defining function. 
	
\begin{lem}\label{ConcaveDefFn}
For $0<\al<2/(N+1)$, the function $\phi(x)=\left(\prod_{1\leq i \leq N} l_i(x)\right)^\alpha$ is strictly concave in $P$.
\end{lem}

Assuming the lemma, we let $\tilde{f}$ be any smooth extension of the boundary values to the interior bounded by the values of $f$ and with Hessian bounded by $f_{tt}$ and let $v(x) = \tilde{f}+A(-\phi(x))$. Then $v=f$ on $\d P$ with $\det D^2v \sim l_i^{2\al-2}$ near $l_i=0$ and $\det D^2v \sim (l_i l_j)^{2\al-2}$ near the corner $l_i=l_j=0$, so $v\leq u_i$ for each $i$. Therefore $u_i$ converges uniformly on compact subsets to a continuous convex function $u$  on $\overline{P}$ such that $u=f$ on $\d P$ and $\mu_\varphi(u)$ realizes $\mu$, and $u \in C^\al(\bar{P})$.
\end{proof}

\begin{proof}[Proof of Lemma \ref{ConcaveDefFn}]
Note that for a single corner, one can easily see by direct calculation of the Hessian that the function $((y+\la x)y)^\al$ is concave for $0<\al\leq 1/2$ and strictly concave for $0<\al<1/2$ in the region $\{y+\la x>0\}\cap\{y>0\}$. For the barrier in the whole polygon, we show that the function $\phi(x)$ is strictly concave on any line segment contained in $P$. When restricted to a line parametrized by $t$, we have $$\phi(x) = \left(\prod_{1\leq i \leq N}(a_i+b_i t)\right)^\al,$$ and therefore
\begin{align*}
	\phi_{tt} &= \al \phi\left(\al \left(\sum_{i=1}^N \frac{b_i}{a_i+b_i t}\right)^2 - \sum_{i=1}^N \left(\frac{b_i}{a_i+b_i t}\right)^2\right)\\
	&= \al \phi \left(\al \left(\sum_{i,j=1}^N \left(\frac{b_i}{a_i+b_i t}\frac{b_j}{a_j+b_j t}\right)\right)- \sum_{i=1}^N \left(\frac{b_i}{a_i+b_i t}\right)^2\right)\\
	&\leq \al \phi \left(\frac{\al}{2}\sum_{i,j=1}^N\left(\left(\frac{b_i}{a_i+b_i t}\right)^2+\left(\frac{b_j}{a_j+b_j t}\right)^2\right)- \sum_{i=1}^N \left(\frac{b_i}{a_i+b_i t}\right)^2\right)\\
	&= \al \phi \left(\frac{(N+1)\al}{2}-1\right)\sum_{i=1}^N \left(\frac{b_i}{a_i+b_i t}\right)^2, 
\end{align*}
which is negative if $0<\al<\frac{2}{N+1}$. 

\end{proof}

Hence for every choice of values at the vertices $\{\al_k\}$, there exists a  unique continuous convex solution $u$ to the Dirichlet problem (\ref{Dir}), which is H\"{o}lder continuous of exponent $\al$ for any $\al <2/(N+1)$ at the boundary. Restricting this solution to any uniformly convex subdomain that does not touch the boundary, we have a solution of a Monge-Amp\`{e}re equation with uniformly bounded right-hand side, so by the results of Cafarelli and Guti\'{e}rrez \cite{CG}, the solution is in fact smooth in the interior.
\medskip

\section{Behavior Near the Interior of an Edge}

Now we investigate the behavior of the solution $u$ at the boundary near an edge and away from the vertices. We take our edge to be a segment of $\{y=0\}$ containing an interval around $(0,0)$. Our goal is to show that in a small half-disc $B_r^+(0)$, $u = y\log y + f$, where $f \in C^\infty(\overline{B_r^+(0)})$. The main technique is the partial Legendre transform as in \cite{DS}, which is useful in dimension two, where the transformed function satisfies a quasilinear equation. (Another way to understand why the dimension two case is simpler, without reference to the partial Legendre transform, is as follows: In dimension two, the second tangential derivative $u_{xx}$ a solution to linear equation with 0 right-hand side, so it is possible to obtain a positive lower bound for $u_{xx}$.)  
\medskip

As a model, consider the degenerate Monge-Amp\`{e}re equation
\begin{equation}\label{halfplaneMA}
	\det D^2u = \frac{1}{y} \text{ in }\R^2_{y>0},\quad u|_{y=0} = \half x^2.
\end{equation}
We perform a partial Legendre transform in the $x$-variable as follows. Let
\begin{equation}
	(p,q) = (u_x,y), \qquad u^*(p,q) = xu_x(x,y)-u(x,y).
\end{equation}
Under this change of variables we find that $u^*$ satisfies the equation
\begin{equation}
	\frac{1}{y}u^*_{pp} + u^*_{yy} = 0
\end{equation}
with boundary data $$u^*(p,0) = \left(\half x^2\right)^* = \half p^2$$ along the flat boundary. Additionally, the partial Legendre transform of a solution of the Monge-Amp\`{e}re equation is necessarily convex in the tangential, or $p$-direction, and concave in the $y$-direction. A model solution of this equation is $u^*(p,y) = \half p^2 - y\log y$. 

\smallskip
Now let us consider the problem
\begin{equation}\label{linMA}
\begin{cases}
\frac{1}{y}u_{pp}+u_{yy} = 0 & \text{in $B_1^+(0)$}\\
u = g & \text{on $\d B_1^+(0)$}\\
\end{cases}
\end{equation}
in a half-ball $B_1^+(0) = B_1(0)\cap\{y>0\}$ with arbitrary boundary data. We begin by establishing the existence of solutions to this equation by approximating by solutions to uniformly elliptic equations. We use a Bernstein technique to control the derivatives; the effect of the degeneracy is that we can only control the derivatives in the direction parallel to the edge.

\begin{prop}
	For any $g \in C^0(\partial B_1^+(0))\cap C^4(\{y=0\}\cap \overline{B_1^+})$, there is a unique strong solution $u$ of (\ref{linMA}) in the sense that $u \in C^2(B_1^+(0))\cap C^\alpha(\overline{B_1^+(0)})$ for all $\alpha < 1$ and $u=g$ on $\partial B_1^+(0)$. Furthermore,

\begin{alignat}{2}\label{maxprin}
	\max_{\substack{\overline{B_1^+}}}u \leq \max_{\substack{\partial B_1^+}}g, &\qquad	 \min_{\substack{\overline{B_1^+}}}u \geq 		  \min_{\substack{\partial B_1^+}}g,
\end{alignat}
and 
\begin{equation}
	\max_{\substack{\overline{B_{1/2}^+}}}|u_p| \leq C(\| g\| _{C^2}), \qquad \max_{\substack{\overline{B_{1/2}^+}}}|u_{pp}| \leq C(\| g\| _{C^4}).
\end{equation}
\end{prop} 

\begin{proof}
For $\epsilon > 0$ sufficiently small, let $\eta_{\epsilon}(y) \in C^{\infty}(-\infty, \infty)$ such that $$\eta_{\epsilon}(y)=1/y \text{ for } y>2\epsilon;\quad \eta_{\epsilon}(y) = 1/\epsilon \text{ for } y \leq \epsilon.$$ By the standard theory of uniformly elliptic equations, the equation
\begin{equation}
	L_{\epsilon} u := \eta_\epsilon u_{pp} + u_{yy}  = 0 \text{ in } B_1^+, u = g \text{ on } \partial B_1^+,
\end{equation}
has a unique solution $u^\epsilon \in C^2(B_1^+)\cap C^{\alpha}(\overline{B_1^+})$. By the maximum principle, $u^\epsilon$ satisfies \eqref{maxprin}. We will find uniform estimates for $u^\epsilon$ and take $\epsilon \rightarrow 0$ to obtain the desired solution.

\smallskip
In this setting, it is important to establish that the solution is continuous up to the boundary. While it is clear that any limit of the $u^\ep$ will satisfy \eqref{maxprin}, without any better control than the $L^\infty$ norm there is nothing to prevent the graph of the limit from becoming vertical on $\{y=0\}$, which is to say that possibly $\lim_{y \rightarrow 0} u(p,y)\neq g(p)$. To see that we will have continuity to the prescribed boundary values we construct barriers as follows. For each point $(p_0,0)$ with $-1 < p_0 < 1$, let $P_{p_0}^+(p)$ and $P_{p_0}^-(p)$ be the tangent parabolas to $g$ at $p_0$ opening up and down, respectively. Set

\begin{eqnarray*}
	v_{p_0,\delta}^+(p,y) &:= P_{p_0}^+(x) - By\log y  + Cy + \delta,\\
	v_{p_0,\delta}^-(p,y) &:= P_{p_0}^-(x) + B'y\log y - C'y - \delta,
\end{eqnarray*}
with B, B', C, C' positive constants to be chosen below, and $\delta > 0$ small. We compute $$L_\epsilon v_{p_0,\delta}^+(p,y) = \eta_\epsilon A - B/y < 0$$ for $B>A$, where $A>0$, the quadratic coefficient in $P_{p_0}^+$ such that the parabola lies above $g$ on $\{y=0\}$, depends on $||g||_{C^2}$, and $C$ is chosen large enough, depending on $\| g\| _{C^0}$, so that $v_{p_0,\delta}^+$ lies above $g$ on the half-circle. The function $v_{p_0,\delta}^+$ is thus a supersolution that lies above the solution $u^\epsilon$ for each $p_0$ and each $\delta$. Similar considerations for $v_{p_0,\delta}^-$ give that $$v_{p_0,\delta}^-(p,y) \leq u^\epsilon(p,y) \leq v_{p_0,\delta}^+(p,y)$$ for all $\delta > 0$, therefore $$g(p_0,0) + B'y\log y - C'y \leq u^\epsilon(p_0,y) \leq g(p_0,0) - By\log y + Cy,$$ and $$|u^\epsilon(p_0,y) - g(p_0)| \leq |Dy\log y|$$ for $D$ independent of $\epsilon$.

\smallskip
We may now take $\epsilon$ to 0 to obtain a solution $u$ of \eqref{linMA} in $C^2(B_1^+(0))\cap C^\alpha(\overline{B_1^+(0)})$, which is unique since $u$ satisfies \eqref{maxprin} since solutions of our equation can have no interior maxima or minima. 

\smallskip
Now we use the same argument as in \cite{DS} to obtain a bound on $u_x$. We show
\begin{equation}
L(Cu^2 + \varphi^2 u_p^2) \geq 0
\end{equation}
for a solution $u$ and a cutoff function $\varphi$ where $\varphi = 1$ in $B_{1/2}^+$, $\varphi = 0$ in $B_1^+\setminus B_{3/4}^+$, and $\varphi_y = 0$ for all $y \leq 1/4$. We compute $$L(u^2)=2(u_p^2/y+u_y^2)$$ and 
\begin{multline}\label{derivbound}
 	L (Cu^2 + \varphi^2 u_{p}^2) = 2C (u_{p}^2/y + u_{y}^2) + 2\varphi^2(u_{pp}^2/y + u_{py}^2) + L(\varphi^2)u_p^2\\
	+ 8(\varphi_x u_y)(\varphi u_{pp})/y + 8(\varphi_y u_p)(\varphi u_{py}).
\end{multline}
Since also we may assume
\begin{equation*}
L(\varphi^2) \geq -C_1/y, \quad |\varphi_y u_p| \leq C_1 |u_p|/y^{1/2},
\end{equation*}
so $$8(\varphi_y u_p)(\varphi u_{py}) \geq -C_1^2 u_p^2/y - \varphi^2 u_{py}^2,$$ and similarly for the other mixed term, we see that the right-hand side of \eqref{derivbound} can be made non-negative in $B_1^+$ for sufficiently large $C$. Hence $$\sup_{B_{1/2}^+}|u_p| \leq C^{1/2}\sup_{B_1^+}|u| +\sup_{\{y=0\}}|g_p|^2.$$ Further, since $L u_p =0$, the same argument implies that $|u_{pp}| \leq C$ in $B_{1/2}^+$, using the regularity of $g$ on the flat boundary. 

\end{proof}

We shall have need for similar estimates (depending additionally on $b,c,f$) for the more general equation
\begin{equation}
	u_{pp}+yu_{yy}+bu_p+cu=f
\end{equation}
where $b,c,f$ are bounded functions. These follow from the same arguments as above. Note that the sign of $c$ does not matter here, since barriers of the form $v = Cy\log y$ are bounded, go to $0$ as $y$ goes to $0$, and have $yv_{yy}=C$.

\medskip
We cannot perform a proper Taylor expansion at a point on the flat boundary since we expect that $|u^\epsilon_y(p,0)|$ will go to infinity as $\epsilon$ goes to zero. Nevertheless it is still true that
\begin{eqnarray*}
	u(p,y) &=& u(p,0) + \int_0^y u_y(p,s)ds \\
				   &=& u(p,0) + \int_0^y \left(C_\delta + \int_\delta^s u_{yy}(p,t)dt\right)ds\\
				   &=& u(p,0) + \int_0^y \left(C_\delta + \int_\delta^s -\frac{u_{pp}(p,t)}{t}dt\right)ds,\\
\end{eqnarray*}
and since $u_{pp}$ solves the same equation as $u$, we have that $u_{pp}(p,t) = u_{pp}(p,0) + O(t\log t)$ (this requires $g \in C^4$ on the flat boundary). Hence
\begin{eqnarray*}
	u(p,y) &=& u(p,0) + \int_0^y \left(C_\delta + \int_\delta^s -\frac{u^\epsilon_{pp}(p,0)+O(t\log t)}{t} dt\right)ds\\
					&=& u(p,0) + \int_0^y \left(-u_{pp}(p,0)(\log s -\log(\delta) + O(s \log s) + C\right)ds\\
					&=& u(p,0) + -u_{pp}(p,0)y\log y + u_{pp}(p,0)(1-\log(\delta))y + Cy + O(y^2\log y)\\
					&=& u(p,0) -u_{pp}(p,0)y\log y + O(y).					
\end{eqnarray*}

Now for each $n$, $\d_p^n u$ solves the same equation, so as long as the boundary function $g$ possesses $n+2$ continuous derivatives along the flat boundary, the same estimates and the same type of expansion will hold for $\d_p^n u$. In particular, 
\begin{equation}
u_{pp} = u_{pp}(p,0) - u_{pppp}(p,0)y\log y +O(y),
\end{equation}
and now we can use the equation $u_{yy} = -u_{pp}/y$ to expand to the next order in y:
\begin{equation}
u(p,y) = u(p,0)-u_{pp}(p,0)y\log y + w(p)y + \half u_{pppp}(p,0)y^2\log y +O(y^2),
\end{equation}
and so on.

\medskip
Now we return to using $u$ to denote the solution of the Monge-Amp\`{e}re equation, and $u^*$ for its partial Legendre transform. If $u$ satisfies the Dirichlet problem for the MA-equation in the polytope, then its partial Legendre transform $u^*$ satisfies
\begin{equation}\label{PLeqn}
	\frac{u^*_{pp}}{\varphi(x,y)} + u^*_{yy} = 0, \qquad u^*_{pp}(p,0) = \frac{1}{u_{xx}(x,0)} = h(x,0) 
\end{equation}
and $u^*_{pp}>0$. Care is needed in understanding this equation: the function $\varphi(x,y)$ depends on $p$ through the Legendre transform, in that at the point $(p,y)$, $x = u_p^*(p,y)$, so the equation has a non-linear dependence on $u_p^*$. If we can show that the solution to this equation has an expansion in terms of $y^n$ and $y^n\log y$ like the solution to the model equation, then we can use the boundary condition to determine the coefficient functions. We state the existence of such an expansion as a lemma:
 
\begin{lem}
If $u^*$ solves (\ref{PLeqn}), then in some $Q_\ep$, for each $k\in \mathbb{N}$, $u^*$ has an expansion along the boundary as
\begin{equation}\label{expansion}
	u^*(p,y) = u^*(p,0) + \sum_{i=1}^k \frac{1}{i!}\hat{u}^*_i(p)y^i \log y +\sum_{i=1}^{k-1}\frac{1}{i!}u^*_i(p)y^i + O(y^k).
\end{equation}
\end{lem}

Now we can prove the main theorem of this section, which completes the proof of the regularity stated in Theorem \ref{maintheo}, part (b).

\begin{theo}
Let $u$ be the solution of the Dirichlet problem (\ref{Dir}), with right-hand side given by (\ref{cond5}) and boundary data given by (\ref{bdrycond}), and suppose $l_1(x,y) = y$. Then for any point $q$ on the edge $\{y=0\}\cap \d P$, there is a small half ball around $q$ such that $u(x,y) = y\log y +f(x,y)$, with $f \in C^\infty(\overline{B}^+_r(q))$.
\end{theo}

\begin{proof}
Since $u$ satisfies $u_{xx} = 1/h(x,0)$ on $\{y=0\}$, its partial Legendre transform $u^*$ satisfies the equation (\ref{PLeqn}) with the boundary condition. Assuming the lemma, $u^*$ has an expansion as in (\ref{expansion}). Now we use (\ref{PLeqn}) to write
\begin{eqnarray*}
	u^*(p,y) &=& - \int^y \int^{y'} \frac{u^*_{pp}(p,y'')}{\varphi(x,y'')}dy''dy'\\
		     &=& - \int^y \int^{y'} \frac{u^*_{pp}(p,y'')/h(x,y'')}{y''}dy''dy'\\
		     &=& - \int^y \int^{y'} \frac{(u^{*}_{0}{''}(p)+\hat{u}^{*}_{1}{''}(p)y \log y +...)(1/h(x,0) + a_1(x)y''+...)}{y''}dy''dy',
\end{eqnarray*}
where we have also used a polynomial expansion for $1/h(x,y)$. We compute the $y^i \log y$ terms of the expansion explicitly: The $y\log y$ term can only come from two integrations of $1/y$, which only occurs in the very first term in the expansion, so
\begin{equation}
	\hat{u}^*_1(p) = -u^*_0{''}(p)/h(x,0) \equiv -1
\end{equation}
by the boundary condition. Similarly we can compute the higher coefficients:
\begin{eqnarray}
	\hat{u}^*_2(p) &=& -\hat{u}^*_1{''}(p)/h(x,0) \equiv 0,\\
	\hat{u}^*_3(p) &=& -\hat{u}^*_2{''}(p)/h(x,0) - \hat{u}^*_1{''}(p)a_1(x) \equiv 0,
\end{eqnarray}
and in the same way, all the higher coefficients on $y^n\log y$ terms are identically 0. Thus the solution of (\ref{PLeqn}) with the particular boundary data, is of the form
\begin{equation}
	u^*(p,y) = u^*(p,0) - y\log y + \sum_{i=1}^N \frac{1}{i!}u^*_i(p)y^i + o(y^N).
\end{equation}
We obtain the theorem by taking the partial Legendre transform back to $u$. Note that this theorem did not use any of the prescribed data of the function $h$ at the vertices.
\end{proof}

It remains only to show that the partial Legendre transform $u^*$ possesses such an expansion, which will be established by a perturbation argument. 

\smallskip
Assume for simplicity that $a(0,0)=1$. Let $Q_r = \{y \leq r-x^2\} \cap \{y \geq 0\}$.
\begin{prop}\label{perturbArg}
Let $u$ solve
\begin{equation}\label{perturb}
	u_{pp}+ya(p,y)u_{yy}+bu_p+cu=f, \qquad u|_{\d Q_1} = g,
\end{equation}
where $a,b,c,f\in C(\overline{Q_1})$, $g\in C^2(\{y=0\})$. Then there exists $\ep>0$ such that if
\begin{equation}\label{weightedHolder}
 \frac{|1-a(r^{1/2}p,ry)|}{r^\al}<\ep
\end{equation}
and
\begin{equation}
\frac{|f(r^{1/2}p,ry)-f(0,0)|}{r^\al}<\ep,
\end{equation}
then $u(p,y) = u(p,0) + u_1(p)y\log y + u_2(p)y+o(y^{1+\al})$.
\end{prop}

\begin{proof}
Let $w$ be the solution of the constant coefficient model equation
\begin{equation}
	w_{pp} +yw_{yy}+bw_p+cw = f, \qquad w|_{\d Q_{3/4}} = u.
\end{equation}
Then in $Q_{1/2}$,
\begin{equation}
	|(u-w)_{pp}+ya(p,y)(u-w)_{yy}| = |y(1-a)w_{yy}| \leq \ep y w_{yy} \leq C\ep,
\end{equation}
from which we see that $u-w$ solves the equation with small right-hand side and zero boundary data. Comparing this function to a barrier $v = C\ep y \log y$, we obtain
\begin{equation}
	\max_{Q_{1/2}}|u-w| \leq C'\ep.
\end{equation}
We set
\begin{equation}
P_1(p,y):= w_0(p)+w_1(p)y\log y+w_2(p)y,
\end{equation}
the first terms in the expansion for $w$. We have $|w-P_1|=|w-(w(p,0)+w_1(p)y\log y+w_2(p)y))|\leq Cy^{1+\al}$ for all $\al<1$, and therefore
\begin{equation}
	|u-P_1| \leq C\ep+Cr^{1+\al} \leq C'r^{1+\al'}
\end{equation}
in $Q_r$, $r$ small, for $\ep \leq cr^{1+\al}$. 

\smallskip
Now we iterate this comparison, using the scaling of the equation. Set
\begin{eqnarray}
	\tilde{u}(p,y) &=& \frac{(u-P_1)(r^{1/2}p,ry)}{r^{1+\al}},\\
	\tilde{L}v &=& v_{pp} + ya(r^{1/2}p,ry)v_{yy}+r^{1/2}bv_p+rcv,\\
	\tilde{f}(p,y) &=& \tilde{L}\tilde{u}(p,y),
\end{eqnarray}
Then
\begin{align*}
|\tilde{f}| =& \bigg| \frac{f(r^{1/2}p,ry)}{r^\al} - \bigg[ r^{-\al}w_0{''}(r^{1/2}p) + w_1{''}(r^{1/2}p)r^{1-\al}y\log(ry) +w_2''(r^{1/2}p)r^{1-\al}y \\ & + r^{-\al}ya(r^{1/2}p,ry)(\frac{w_1}{y})\\ & + r^{-1/2-\al}b\bigg(r^{1/2}w_0'(r^{1/2}p) +r^{3/2}w_1'(r^{1/2}p)y\log(ry) + r^{3/2}w_2'(r^{1/2}p)y \bigg) \\ & + r^{-\al}c\left(w_0(r^{1/2}p)+rw_1(r^{1/2}p)y\log(ry) + rw_2(r^{1/2}p)y\right) \bigg] \bigg| \\
						    \leq& \left|\frac{(1-a(r^{1/2}p,ry))w_1}{r^\al}\right|\\& + \left|\frac{f(r^{1/2}p,ry)}{r^\al}-r^{-\al}\left(w_0{''}(r^{1/2}p)+ w_1(r^{1/2}p)+bw_0'(r^{1/2}p) + cw_0(r^{1/2}p)\right)\right|\\ & + Cr^\beta\\
						    \leq& C\ep + C\left|\frac{f(r^{1/2}p,ry)-f(r^{1/2}p,0)}{r^\al} \right| +Cr^\beta\\
						    \leq& C'\ep.					     	
\end{align*}
In the third line, we have combined all terms that go as a positive power of $r$ in $Cr^\beta$, and in the fourth line the constant $C'$ depends only on the $L^\infty$ norms of the coefficients. So we can compare $\tilde{u}$ to $\tilde{w}$, the solution of $w_{pp}+yw_{yy} + r^{1/2}bw_p+rcw = \tilde{f}$ in $Q_{r}$ matching $\tilde{u}$ on the boundary. Again, $\tilde{w} = \tilde{w}_0(p) + \tilde{w}_1(p)y\log y + \tilde{w}_2(p)y+o(y^{1+\al})$, where $\tilde{w}_1, \tilde{w}_2$ depend on $\tilde{f}$ ($\tilde{w}_0 = 0$ since $\tilde{w}=0$ on the flat boundary). We can thus iterate the comparison, since $\tilde{f}$ also satisfies $|\tilde{f}(r^{1/2}p,ry)-\tilde{f}(0,0)|<Cr^\al$. For example,
\begin{equation}
\left|\frac{u(r^{1/2}p,ry)-P_1(r^{1/2}p,ry)}{r^{1+\al}} - \left(\tilde{w}_1(p)y\log y+\tilde{w}_2(p)y\right)\right|_{L^{\infty}(Q_{r})} < Cr^{1+\al}
\end{equation}
so
\begin{equation}
\left|u(p,y)-P_1(p,y)-\left(r^\al\tilde{w}_1(p/r^{1/2})y\log(y/r) +r^\al\tilde{w}_2(p/r^{1/2})\right)\right|_{L^{\infty}(Q_{r^2})} < Cr^{2(1+\al)}
\end{equation}
 We obtain for each $k$ a function $P_k = a_k(p) +b_k(p)y\log y$ with coefficient functions bounded by the $Cr^{k\al}$, where C depends on the $L^\infty$ norms of the original coefficients. Therefore the constant $C_k$ in the size of the right-hand side $|\tilde{f}^{(k)}|\leq C_k \ep$ remains bounded as $k$ goes to infinity. The sum of the $P_k$ is thus bounded by a convergent geometric series, and so 
\begin{equation}
	\left|u-\sum_{k=1}^n P_k\right|_{L^\infty(Q_{r^n})} \leq Cr^{n(1+\al)},
\end{equation}
and we get the conclusion of the lemma by taking a limit of the $\sum_{k=1}^n P_k$.
\end{proof}

We now show that the hypotheses of the preceding proposition are satisfied by the  partial Legendre transform equation (\ref{PLeqn}) for $u^*$, as well as by the equations satisfied by its derivatives $\d_p^ku^*$. First, we must verify that the coefficient function $\varphi(x,y) = ya(p,y)$ satisfies the weighted H\"{o}lder condition (\ref{weightedHolder}). If we define the function $\hat{a}(x,y)$ by $\varphi(x,y) = y\hat{a}(x,y)$, then as a function of $(p,y)$, we have
\begin{equation}
a(p,y) = \hat{a}(u^*_p(p,y),y)
\end{equation}
and
\begin{equation}
a(r^{1/2}p,ry) = \hat{a}(u_p^*(r^{1/2}p,ry),ry).
\end{equation}
Since by assumption $\hat{a}$ is smooth, we must show that $u_p^*$ satisfies a similar weighted H\"{o}lder condition.




\begin{lem}
Suppose $u$ satisfies
\begin{equation}
	\det D^2u = \frac{1}{y\hat{a}(x,y)}, \qquad 
\end{equation}
Then there exist $C, \al >0$ such that its partial Legendre transform $u^*$ in $(p,y)$ satisfies
\begin{equation}
\left|u_p^*(r^{1/2}p,ry)-u_p^*(0,0)\right| \leq Cr^\al.
\end{equation}
\end{lem}

\begin{proof}
We may assume that $u_p^*(0,0)=u^*(0,0)=0$. We will exploit the fact that since  $u^*$ arises as the partial Legendre transform of a solution of a Monge-Amp\`{e}re equation, it satisfies the equation (\ref{PLeqn}) and is strictly convex in the $p$-direction. 

\smallskip
There are constants $c,C$ such that 
\begin{equation}
	\half p^2 -cy\log y < u^* < \half p^2 -Cy\log y,
\end{equation}
or
\begin{equation}
	|u^*(p,y_0)-\half p^2| < Cy_0\log y_0.
\end{equation}
in each slice of fixed $y_0$ small. Since $u^*(p,y_0)$ is convex in $p$, we have
\begin{equation}
	|u^*_p(p,y_0) - p| < C\sqrt{y_0 \log y_0}
\end{equation}
for $|p|<1/2$ since this is the largest the derivative can deviate from the derivative of $p^2/2$ before it must get farther away from $p^2/2$ than $Cy_0\log y$. It follows that for any $\al <1/2$, $|u^*_p(r^{1/2}p_0,ry_0)|<Cr^\al.$
\end{proof}


Hence the partial Legendre transform $u^*$ satisfies an equation of the form of (\ref{perturbArg}), and therefore
\begin{equation}
	u^*(p,y) = u^*(p,0) + u^*_1(p)y\log y + u^*_2(p)y+ o(y^{1+\al}).
\end{equation}
Now we differentiate and examine the equations satisfied by the derivatives $\d_p^k u^*$:
\begin{equation}\label{u_p eqn}
	(u^*_p)_{pp} + y\hat{a}(x,y)(u^*_p)_{yy} + \hat{a}_{x}(x,y)yu^*_{yy}(u^*_p)_p =0,
\end{equation}
and
\begin{equation}\label{u_pp eqn}
	(u^*_{pp})_{pp} + y\hat{a}(x,y)(u^*_{pp})_{yy}+\hat{a}_x(x,y)yu^*_{yy}(u^*_{pp})_p \\
	+ \left(\hat{a}_x(x,y)yu^*_{yyp}+\hat{a}_{xx}(x,y)yu^*_{yy}u^*_{pp}\right)(u^*_{pp}) = 0.
\end{equation}
We can then obtain that $u^*$ has an expansion as follows: Since $yu^*_{yy}$ is bounded and $\hat{a}(x,y)$ is smooth, the equation satisfied by $u^*_p$, (\ref{u_p eqn}), satisfies the hypotheses of the proposition, and so $u^*_p$ also admits such an expansion to order $o(y^{1+\al})$. It follows that $yu^*_{yyp}$ is also bounded, and so is $u^*_{pp}$, so we may also apply the proposition to the equation (\ref{u_pp eqn}), and thus
\begin{equation}
	u^*_{pp}(p,y) = u^*_{pp}(p,0)+u^{*}_1{''}(p)y\log y + u^{*}_2{''}(p)y + o(y^{1+\al}).
\end{equation} 
Since $yau^*_{yy} = -u^*_{pp}$, the next two terms in the expansion for $u^*$ must be of the form $u^*_3(p)y^2\log y$ and $u^*_4(p)y^2$. Then $u^{*}_3{''}(p)y^2\log y$ and $u^{*}_4{''}(p)y^2$ are the next two terms in $u^*_{pp}$, and we can continue this process indefinitely.

\end{document}